\documentclass[10pt,a4paper, xcolor]{article}
\usepackage{amsfonts}
\usepackage{amssymb,amsthm, amsmath,graphicx,bm,ebezier,amsthm}
\usepackage{color,enumerate,lineno}
\usepackage{url}

\usepackage{amsfonts}
\usepackage{amssymb,amsthm, amsmath,graphicx,bm,ebezier,amsthm,mathtools}
\usepackage{color,enumerate}
\usepackage{url,lineno}

\usepackage[latin1]{inputenc}

\newtheorem{theorem}{Theorem}
\newtheorem{definition}[theorem]{Definition}
\newtheorem{proposition}[theorem]{Proposition}
\newtheorem{corollary}[theorem]{Corollary}
\newtheorem{lemma}[theorem]{Lemma}

\theoremstyle{remark}

\newtheorem{example}[theorem]{Example}
\newtheorem{algorithm}[theorem]{Algorithm}
\newtheorem{remark}[theorem]{Remark}

\renewcommand{\L}{\mathcal{L}}

\renewcommand{\v}{\overrightarrow}

\def\lcm{{\rm lcm}}

\def\E{{\sf Z}}

\def\N{\mathbb{N}}
\def\R{\mathbb{R}}
\def\Z{\mathbb{Z}}
\def\Q{\mathbb{Q}}

\def\int{\mathrm{int}}

\newcommand{\V}{\textsf V}
\renewcommand{\H}{\mathcal H}
\newcommand{\F}{\mathcal F}
\renewcommand{\l}{{\tt L}}

\title{Computation of Delta sets of numerical monoids}

\author{J. I. Garc\'{\i}a-Garc\'{\i}a\footnote{Departamento de Matem\'aticas, Universidad de C\'adiz,
E-11510 Puerto Real (C\'{a}diz, Spain). E-mail: ignacio.garcia@uca.es. Partially supported by MTM2010-15595 and Junta de Andaluc\'{\i}a group FQM-366. }\\
M. A. Moreno-Fr\'{\i}as\footnote{Departamento de Matem\'aticas, Universidad de C\'adiz,
E-11510 Puerto Real (C\'{a}diz, Spain). E-mail: mariangeles.moreno@uca.es. Partially supported by MTM2008-06201-C02-02 and Junta de Andaluc\'{\i}a group FQM-298.}\\
A. Vigneron-Tenorio\footnote{Departamento de Matem\'aticas, Universidad de C\'adiz,
E-11406 Jerez de la Frontera (C\'{a}diz, Spain). E-mail: alberto.vigneron@uca.es. Partially supported by MTM2012-36917-C03-01 and Junta de Andaluc\'{\i}a group FQM-366.}\\
}

\date{}

\begin{document}

\maketitle

\begin{abstract}
Let  $\{a_1,\dots,a_p\}$ be the minimal generating set of a numerical monoid $S$.
For any $s\in S$, its Delta set is defined by $\Delta(s)=\{l_{i}-l_{i-1}\mid i=2,\dots,k\}$ where $\{l_1<\dots<l_k\}$  is the set $\{\sum_{i=1}^px_i\mid s=\sum_{i=1}^px_ia_i \textrm{ and } x_i\in \N \textrm{ for all }i\}.$
The Delta set of a numerical monoid $S$, denoted by $\Delta(S)$,  is the union of all the sets $\Delta(s)$ with $s\in S.$
As proved in \cite{Chapman-Hoyer-Kaplan},  there exists a bound $N$ such that
$\Delta(S)$ is the union of the sets $\Delta(s)$ with $s\in S$ and $s<N$.
In this work, we obtain a sharpened bound
and we present an algorithm for the computation of $\Delta(S)$ that requires only the factorizations of  $a_1$ elements.


\smallskip
{\small \emph{MSC-class:} 20M14 (Primary),  20M05 (Secondary).}
\end{abstract}

\section*{Introduction}
The study of the structure of $\Delta(S)$ and its computation plays an important role in the theory of  non-unique factorization.
For example in \cite{Bowles-Chapman-Kaplan-Reiser}, it found a rigorous study of $\Delta(S)$ for  numerical monoids that shows the structure of $\Delta(S)$ can be very complex even in the case $S$ is generated by only three elements.
Also in \cite{Bowles-Chapman-Kaplan-Reiser},
some bounds for the maximum and the minimum of $\Delta(S)$ with $S$ a numerical monoid are given.
Another interesting work is \cite{Chapman-Garcia-Llena-Malyshev-Steinberg} where
some results concerning the structure
of the Delta sets
of BF-monoids are proved  and
it is shown
that the minimum and the maximum of $\Delta(S)$
can be completely determined using the Betti elements of $S$.
In \cite{Chapman-Malyshev-Steinberg},
the conditions which must be satisfied by the generators of $S=\langle a_1,a_2,a_3\rangle$ for $\Delta(S)$
being a singleton are shown.
One of the main results used to compute $\Delta(S)$ is given in \cite{Geroldinger};
it proves that every commutative cancellative reduced atomic monoid $S$ satisfies  that $\min(\Delta(S))=\gcd(\Delta(S))$.
A method for computing $\Delta(S)$ is found in \cite{Chapman-Hoyer-Kaplan};
in that paper,
it is  proved that
for every  numerical monoid $S$ with minimal system of generators $a_1<\dots<a_p$,
and
for every element $s\in S$ such that $s\geq 2pa_2a_p^2$
it must hold $\Delta(s)=\Delta(s+a_1a_p)$.
Thus, for a primitive numerical monoid $S$ we have $\Delta(S)=\cup_{s\in S,s<N}\Delta(s)$ with $N=2pa_2a_p^2+a_1 a_p$, and
this implies that the computation of $\Delta(S)$ requires only a finite number of steps.

A different approach
is to study $\Delta(S)$  in different types of  monoids.
For instance
in \cite{Bowles-Chapman-Kaplan-Reiser}, it is proved that if $a_1$ and $a_2$ are integers satisfying $1<a_1<a_2$ and $\gcd(a_1,a_2)=1$, then $\Delta(\langle a_1,a_2\rangle)=\{a_2-a_1\}$.
If $S$ is a numerical monoid generated by a $k$-interval, then $\Delta(S)=\{k\}$, and if $S=\langle n,n+k,(k+1)n-k\rangle $ where $n\geq 3$, $k\geq 1$ and $\gcd(n,k)=1$, then $\Delta(S)=[k,2k,\dots,\lfloor \frac{n+k-1} {k+2}\rfloor k]$.
In \cite{Chapman-Malyshev-Steinberg}, the elements of the set  $\Delta(S)$  with $S=\langle a_1,a_2,a_3\rangle$ a numerical monoid  are characterized, and in the case $a_1=3$ it is proved that
$\left\{ \frac{a_2+a_3}{3}-2\right\}\subseteq \Delta(S) \subseteq [1,\frac{a_2+a_3}{3}-2]\cap \N$.
In \cite{Chapman-Kaplan-Lemburg-Niles-Zlogar}, it is  shown that for an increasing sequence $r_1,\dots,r_t$ of positive integers, a positive integer $n$ and $S_n=\langle n,n+r_1,\dots,n+r_t\rangle$ a numerical monoid, there exists a positive integer $N$ such that  if $n>N$, then $|\Delta(S_n)|=1$.
Other works in this area may be found in
\cite{Baginski-Chapman-Schaeffer,Chapman-Daigle-Hoyer-Kaplan,Geroldinger-Halter,
Haarmann-Kalauli-Moran-oneill-pelayo,Halter-Koch}.

Despite the amount of existing works, the computation of $\Delta(S)$ for a given numerical monoid is not an easy task. The main problems are the high values of the bounds and
the large amount of factorizations that are required even in the cases the bound  is low.
In this work we cover some gaps in the knowledge of the Delta sets of numerical monoids.
We give explicit bounds that improve the bounds obtained in previous works
and we also use some improvements in the computation of the expressions of some elements.
All these advances allow us to get a better algorithm to  compute  the Delta sets of numerical monoids. The theoretical results of this work are complemented with the software \cite{programita} developed in {\tt Mathematica} that  provides us functions to compute the Delta set of a numerical monoid.

The contents of this paper are organized as follows.
In Section \ref{preliminaires}, we introduce some definitions and  notations used in this work. In Section \ref{sec:principal},
we study  the structure of $\Delta(S)$.
These results are used to get the existence of the bound $N_S$.
In Section \ref{sec:formulas}, a formulation of $N_S$ is given. Finally, in Section \ref{sec:compativa}, we give an algorithm to compute $\Delta(S)$, and we illustrate our method with some examples showing their execution times.

\section{Preliminaries}\label{preliminaires}
Let $\N$ be the set of nonnegative integers and let
$\Q_{\geq}$ be the set of nonnegative rational numbers.
If $S$ is an additive submonoid of  $\N$, then $S$ is called a numerical monoid.
We say that the integers $a_1,\dots,a_p$ with $p\in\N\setminus\{0\}$ generate $S$ if $S=\{x_1a_1+\dots+x_pa_p\mid x_i\in \N \textrm{ for all }i=1,\dots,p\}$; this is denoted  by $S=\langle a_1,\dots, a_p\rangle$.
It
is well known
 that the minimal (in terms of cardinality and set inclusion) generating set of $S$ is unique.
In the sequel, we assume  that $\{ a_1,\dots, a_p \}$ is the minimal generating set of $S$ and
$a_1<\dots<a_p$.
A numerical monoid  $S=\langle a_1,\dots, a_p\rangle$ is primitive when
$\gcd( a_1,\dots, a_p )=1$; these monoids are also known as numerical semigroups and
 every numerical monoid is isomorphic to a primitive numerical monoid. Hence, we can narrow our study to the primitive case.
If $S$ is a primitive numerical monoid, then there exists an integer $\F(S)\notin S$ such that $s>\F(S)$ implies that $s\in S$. This integer is known as the Frobenius number of $S$.
For more details on  numerical monoids, the  reader is directed to the monograph \cite{Rosales-Garcia-semigrupos-numericos}.

For numerical monoid $S=\langle a_1,\dots,a_p\rangle$, it must hold $S\cong\N^p/\sim_M$, where $M$ is the subgroup of $\Z^p$ of rank $p-1$ defined by the equation $a_1 x_1 + \dots + a_p x_p=0$
and $\sim_M$ is defined as $x\sim_M y$ if and only if $x-y\in M$ for all $x,y\in\N^p$
(see \cite{RosalesLibro} for further details).
Denote by $\E(s)$ the set $\{(x_1,\dots,x_p)\in\N^p\mid \sum_{i=1}^px_ia_i=s\}$ for every $s\in \N$.
For  all $x,y\in\N^p$ and every $s\in S$, two elements $x$, $y$ belong to $\E(s)$ if and only if $x\sim_M y$.
Define
the linear function $\l:\Q^p\to \Q$ with $\l(x_1,\dots,x_p)=\sum_{i=1}^px_i$.

\begin{definition}
Given
$s\in S$ and $S=\langle a_1,\dots,a_p\rangle$, set
$
 \L(s)=\{ \l(x_1,\dots,x_p)\mid (x_1,\dots,x_p)\in\E(s)
\}
$,
which is known as the set of lengths of $s$ in $S$.
Since $S$ is a numerical monoid, it is not hard to prove that
 this set of lengths is bounded, and so there exist  some positive integers
$l_1<\dots<l_k$ such that
$\L(s)=\{l_1,\dots, l_k\}$. The set
\[
\Delta(s)=\{l_i-l_{i-1}: 2\leq i \leq k\}
\]
is known as the Delta set of s. We globalize the notion of the Delta
set by setting
\[
\Delta(S)=\bigcup_{s\in S} \Delta(s).
\]
The set $\Delta(S)$ is called the Delta set of $S$.
\end{definition}

\section{The structure of $\Delta(S)$}\label{sec:principal}

The computation of $\Delta(S)$ with $S$ a numerical monoid generated by two elements is solved in \cite{Bowles-Chapman-Kaplan-Reiser}. Hence,
we only consider primitive numerical monoids minimally generated by at least three elements.
Denote by $\{e_1,\dots,e_p\}$ the canonical basis of $\R^p$.

\begin{lemma}\label{lema:d}
Let $S=\langle a_1,\dots, a_p\rangle\cong\N^p/\sim_M$ be a numerical monoid.
Then
$\min (\Delta(S))=\min\{\l(m)\mid \l(m)>0,~m\in M\}$.
Furthermore, if $M=\langle m_1,\dots,m_{p-1}\rangle$, then $\min (\Delta(S))=\gcd(\l(m_1),\dots,\l(m_{p-1}))$.
\end{lemma}
\begin{proof}
Since $a_2e_1-a_1e_2\in M$ and $\l(a_2e_1-a_1e_2)=a_2-a_1>0$, then $\{\l(m)\mid \l(m)>0,~m\in M\}\neq \emptyset$.

For every $l\in \Delta(S)$, there exist $s\in S$ and $\gamma,\gamma'\in\E(s)$ such that $l=\l(\gamma)-\l(\gamma')=\l(\gamma-\gamma')$. Since $\gamma-\gamma'\in M$, we obtain that $l\geq \min\{\l(m)\mid \l(m)>0,~m\in M\}$, and thus $\min(\Delta(S))\geq \min\{\l(m)\mid \l(m)>0,~m\in M\}$.

Let  $m$ be an element of $M$ such that $\l(m)>0$. It is easy to find $\gamma\in\N^p$ fulfilling  that $\gamma+m\in\N^p$.
Clearly,
there exists $s\in S$ such that $\gamma,\gamma+m\in\E(s)$.
Since $\l(m)>0$,  using the linearity of $\l$, we have  $\l(\gamma)<\l(\gamma+m)$.
If $\L(s)=\{l_1,\dots,l_t\}$, there exist $1\leq i<j\leq t$ such that
$l_i=\l(\gamma)$ and $l_j=\l(\gamma+m)$. We distinguish two cases: $i+1=j$ and $i+1<j$.
If $i+1=j$, then $\l(m)=\l(\gamma+m)-\l(\gamma)=l_{i+1}-l_i\in\Delta(s)$, and therefore $\min(\Delta(S))\leq \min(\Delta(s))\leq \l(m)$. If $i+1<j$, then $\l(m)=l_j-l_i>l_{i+1}-l_i\in\Delta(s)$, and thus
$\min(\Delta(S))\leq \min(\Delta(s))< \l(m)$.
Thus,
$\min(\Delta(S))\leq\min\{\l(m)\mid m\in M,~\l(m)>0\}$.

We have proved that $\min(\Delta(S))=\min\{\l(m)\mid m\in M,~\l(m)>0\}$.
Since $\{m_1,\dots,m_{p-1}\}$ is a system of generators of $M$,
from the linearity of $\l$
 we deduce that
$\min\{\l(m)\mid m\in M,~\l(m)>0\}=\gcd(\l(m_1),\dots,\l(m_{p-1}))$. This implies that
$\min(\Delta(S))=\gcd(\l(m_1),\dots,\l(m_{p-1}))$.

\end{proof}

On the sequel we denote by $d$ the element $\min(\Delta(S))$.
One of the consequences of  Lemma \ref{lema:d} is that
 $d$  divides
 $\l(\gamma)-\l(\gamma')$
 for every $s\in S$ and every $\gamma,\gamma'\in\E(s)$, and therefore $d$ divides all the elements of $\Delta(S)$.

\begin{definition}\label{d:v}
Under the assumptions of  Lemma \ref{lema:d},
there exist $u_1,\dots,u_{p-1}\in\Z$ such that $d=
u_1 \l(m_1)+\dots+u_{p-1}\l(m_{p-1})$.
We refer to the vectors $\v v = u_1 m_1+\dots+u_{p-1}m_{p-1}\in M$ as the
minimum length increase  vectors.
\end{definition}

\begin{definition}
Under the assumptions of  Lemma \ref{lema:d},
define $\v h=\frac{d}{a_p-a_1}(a_pe_1-a_1e_p) \in\Q^p$.
\end{definition}

Note that  every minimum length increase  vector $\v v$ and  the vector $\v h$ verify $\l(\v v)=\l(\v h)=d$.

Consider $E$ the subgroup of $\Z^p$ defined by the equation $x_1+\dots+x_p=0$ (note that for every $\v c\in E$, $\l(\v c)=0$).
Denote by $\V(E)$ the vectorial subspace of $\Q^p$ defined by the equation  $x_1+\dots+x_p=0$.
It satisfies $E\subset \V(E)$,  $\dim(\V(E))=p-1$ and   $\v v,\v h\not \in \V(E)$.
For every $s\in\N$, let  $\H_s$ be the affine hyperplane $\{(x_1,\dots,x_p)\in\Q^p\mid \sum_{i=1}^px_ia_i=s\}$.
The set $\H_0$ is the vectorial subspace over $\Q$ generated by $M$, its defining equation is  $a_1x_1+\dots+a_px_p=0$ and it verifies $\dim (\H_0)=p-1$.

\begin{definition}
Let $S=\langle a_1,\dots, a_p\rangle$ be a numerical monoid.
 Define
$\v {q_i}=\frac{1}{\gcd(a_i-a_p,-a_1+a_p,a_1-a_i)}((a_i-a_p)e_1+(-a_1+a_p)e_i+(a_1-a_i)e_p)\in\Z^p$
for every $i=2,\dots,p-1$.
\end{definition}
It is straightforward to prove that the vectors $\v {q_i}$ verify the defining equations of $M$ and $E$. Therefore,
$\v {q_i}\in M\cap E$ for every $i=2,\dots,p-1$.
Note that, since the $i$th coordinate of $\v {q_i}$ is greater than zero for all $i=2,\dots,p-2$ (recall that $a_1<\dots<a_p$), the set $\{\v{q_2},\dots,\v q_{p-1}\}$ is $\Q$-linearly independent.

\begin{remark}\label{r:base}
The defining  equations of $\V(E)$ and $\H_0$ are $\Q$-linearly independent. Thus, $\dim(\H_0\cap \V(E))=p-2$. Since $\{\v {q_2},\dots,\v q_{p-1}\}\subset M\cap E\subset\H_0\cap \V(E)$ is a linearly independent set,  this set is a basis of $\H_0\cap\V(E)$.
\end{remark}

\begin{definition}
Let $S=\langle a_1,\dots, a_p\rangle$ be a numerical monoid.
For all $s\in\N$ and for all $i\in\{1,\dots,p\}$, define
$X_i(s)=\frac s {a_i} e_i\in\Q^p_{\geq }$, the point of  intersection of the affine hyperplane $\H_s$ with the $x_i$-axis.
These elements verify  $\l(X_i(s))=s/a_i$ and $\l(X_1(s))>\dots>\l(X_p(s))$.
\end{definition}

\begin{definition}
Let $S=\langle a_1,\dots, a_p\rangle$ be a numerical monoid.
For all $s\in\N$ and for all $i\in\{1,\dots,p\}$, define
$P_{i}(s)=\frac{s(a_i-a_p)}{a_i(a_1-a_p)}e_1+
\frac{s(a_1-a_i)}{a_i(a_1-a_p)}e_p\in\Q^p_{\geq }$.
Clearly, $P_1(s)=X_1(s)$, $P_p(s)=X_p(s)$, $\l(P_i(s))=\l(X_i(s))$  for all $i\in\{2,\dots,p-1\}$ and
 $\l(X_1(s))=\l(P_1(s))>\l(P_2(s))>\dots>\l(P_{p-1}(s))>\l(P_p(s))=\l(X_p(s))$.
\end{definition}

For every $A,B\in\Q^p$, denote by $\overline{AB}$
the set  $\{ A+\lambda \overrightarrow{AB}|\lambda\in\Q,0\leq\lambda\leq 1\}$,
the line segment with endpoints $A$ and $B$.
Note that
the points  $P_i(s)$ can be expressed as
$P_i(s)=\frac{a_1(a_i-a_p)}{a_i(a_1-a_p)}X_1(s) +\frac{a_p(a_1-a_i)}{a_i(a_1-a_p)}X_p(s)$ where  $\frac{a_1(a_i-a_p)}{a_i(a_1-a_p)}+\frac{a_p(a_1-a_i)}{a_i(a_1-a_p)}=1$,
$0\leq \frac{a_1(a_i-a_p)}{a_i(a_1-a_p)}$ and  $0\leq \frac{a_p(a_1-a_i)}{a_i(a_1-a_p)}$. This implies that
$P_i(s)=X_1(s)+\left(1-\frac{a_1(a_i-a_p)}{a_i(a_1-a_p)}\right)\overrightarrow{X_1(s)X_p(s)}$, and thus
$P_i(s)\in \overline{X_1(s) X_p(s)}$.

Denote by $r$ the line defined by $X_1(s)$ and $X_p(s)$. The point $X_1(s)$ belongs to the $x_1$-axis and  $X_p(s)$ to the  $x_p$-axis. This implies that
$\overline {X_1(s) X_p(s)}$ is equal to $r\cap \Q^p_{\geq}$.
For every $s\in\N$, denote $R(s)=P_2(s)+\v h$
and $R'(s)=P_{p-1}(s)-\v h$.
Since $P_2(s),P_{p-1}(s)\in \overline{X_1(s) X_p(s)}\subset \Q^p_{\geq}$ and
$\v h$ is proportional to the vector $\v {X_1(s)X_p(s)}$, the elements
$R(s),R'(s)$ belong to  $r$.

\begin{proposition}\label{pr:existencia}
Let $S=\langle a_1,\dots, a_p\rangle$ be a numerical monoid.
There exists $N_S\in\N$ such that $\overline{R(N_S)R'(N_S)}\subset  \Q^p_{\geq}$ and such that for every  $X\in\overline{R(N_S)R'(N_S)}$ and every $i\in\{2,\dots,p-1\}$ the element
$X+(p-2)\v {q_i}$ belongs to $\Q^p_{\geq}$.
\end{proposition}
\begin{proof}
We have $R(s)=\frac{s(a_2-a_p)}{a_2(a_1-a_p)}e_1+
\frac{s(a_1-a_2)}{a_2(a_1-a_p)}e_p+ \frac{d}{a_p-a_1}(a_pe_1-a_1e_p)$ where   $\frac{s(a_2-a_p)}{a_2(a_1-a_p)}>0$ and $\frac{s(a_1-a_2)}{a_2(a_1-a_p)}>0$ for every $s\in\N\setminus\{0\}$ (recall that $a_1<\dots<a_p$).
 This implies that
there exists $\hat s_1$ such that for all $s\geq \hat s_1$ we have
$\frac{s(a_1-a_2)}{a_2(a_1-a_p)}>\frac{d}{a_p-a_1}a_1$, and thus
$R(s)\in\Q^p_{\geq}$. Similarly, it can be easily proven that  there exists $\hat s_2$ such that
$R'(s)\in\Q^p_{\geq}$ for all $s\geq \hat s_2$. Since
 $\Q^p_{\geq}$ is convex, it follows that
$\overline{R(s)R'(s)}\subset \Q^p_{\geq}$ for all $s\geq \max\{\hat s_1,\hat s_2\}$.

For every $s$,
\[\begin{multlined}
R(s)+(p-2)\v {q_i}=P_2(s)+\v h+(p-2)\v {q_i}\\
=
\frac{s(a_2-a_p)}{a_2(a_1-a_p)}e_1+\frac{s(a_1-a_2)}{a_2(a_1-a_p)}e_p+
\frac{d}{a_p-a_1}(a_p e_1-a_1e_p)\\
+\frac {(p-2)} {\gcd(a_i-a_p,a_p-a_1,a_1-a_i)} ((a_i-a_p)e_1+(a_p-a_1)e_i+(a_1-a_i)e_p)\\
=\left(\frac{s(a_2-a_p)}{a_2(a_1-a_p)}+\frac{da_p}{a_p-a_1}+\frac {(p-2)(a_i-a_p)} {\gcd(a_i-a_p,a_p-a_1,a_1-a_i)}\right)e_1\\
+\left(\frac {(p-2)(a_p-a_1)} {\gcd(a_i-a_p,a_p-a_1,a_1-a_i)}\right)e_i\\
+\left(\frac{s(a_1-a_2)}{a_2(a_1-a_p)}-\frac{da_1}{a_p-a_1}+\frac {(p-2)(a_1-a_i)} {\gcd(a_i-a_p,a_p-a_1,a_1-a_i)}\right)e_p.
\end{multlined}\]
Since $a_1<\dots<a_p$,
the $i$th coordinate of $R(s)+(p-2)\v {q_i}$ belongs to $\Q_{\geq}$.
Furthermore, for every $s\in\N\setminus\{0\}$ we have $\frac{s(a_2-a_p)}{a_2(a_1-a_p)}>0$ and $\frac{s(a_1-a_2)}{a_2(a_1-a_p)}>0$.
Hence,  there exists $s'\in\N$ such that for every $s\geq s'$, $R(s)+(p-2)\v {q_i}\in\Q^p_{\geq}$.
Similarly, it must hold that there exists $s''\in\N$ verifying that $R'(s)+(p-2)\v {q_i}\in\Q^p_{\geq}$ for all $s\geq s''$.
We take $N_S=\max\{s',s'',\hat s_1,\hat s_2\}$. Clearly, the elements
 $R(N_S)$ and $R'(N_S)$ are in $\Q^p_{\geq}$, and for all $i=2,\dots,p-1$,
$R(N_S)+(p-2)\v {q_i},R'(N_S)+(p-2)\v {q_i}\in\Q^p_{\geq}$. Furthermore, for every $X\in \overline{R(N_S)R'(N_S)}$, the element $X+(p-2)\v {q_i}$ belongs to the line segment with endpoints $R(N_S)+(p-2)\v {q_i}\in\Q^p_{\geq}$ and $R'(N_S)+(p-2)\v {q_i}\in\Q^p_{\geq}$. Using again that $\Q^p_{\geq}$ is convex, we obtain that   $X+(p-2)\v {q_i}$ is also in $\Q^p_{\geq}$ for all $i=2,\dots,p-1$.
\end{proof}

The numbers $N_S$ fulfilling the conditions of the above proposition are not unique.
We now define new elements that depend on the election of $N_S$.
For the sake of simplicity in our notation, in the sequel,
we will assume that for a given monoid $S$
the natural number $N_S$  represents an arbitrary fixed element of $S$ fulfilling the conditions of Proposition \ref{pr:existencia}.

\begin{definition}
Let $S=\langle a_1,\dots, a_p\rangle$ be a numerical monoid.
Define the vectors $\v {w}=P_2(N_S)-X_1(N_S)$ and
$\v w'=P_{p-1}(N_S)-X_p(N_S)$.
Note that $\l(\v w)=N_S/a_2-N_S/a_1<0$ and $\l(\v w')=N_S/a_{p-1}-N_S/a_p>0$.
\end{definition}

\begin{lemma}\label{ceros1}
Let $S=\langle a_1,\dots, a_p\rangle$ be a numerical monoid and
let $s\in \N$ be such that $s\geq N_S$. Then
$
\l(X_2(s))\leq \l(X_1(s)+\v w)$ and
$\l(X_p(s)+\v w')\leq \l(X_{p-1}(s))
$.
\end{lemma}
\begin{proof}
Since $s\geq N_S$, we have $s(a_2-a_1)\geq N_S(a_2-a_1)$, and therefore
$s(a_2-a_1)/(a_1a_2)\geq N_S(a_2-a_1)/(a_1a_2)$. Thus,
$s/a_1-s/a_2\geq N_S/a_1-N_S/a_2$, which implies that
$\l(X_1(s)+\v w)=s/a_1+N_S/a_2-N_S/a_1\geq s/a_2=\l(X_2(s))$.
Similarly, we obtain the other inequality.
\end{proof}

The following result  generalizes Proposition \ref{pr:existencia}.

\begin{proposition}\label{pr:p}
Let $S=\langle a_1,\dots, a_p\rangle$ be a numerical monoid and
let $s\in\N$ be such that $s\geq N_S$. Then,
every element $X$ of $\overline{(X_1(s)+\v w+\v h) (X_p(s)+\v w'-\v h)}$  verifies
$X\in\Q^p_{\geq}$ and $X+(p-2)\v {q_i}\in\Q^p_{\geq}$ for all $i\in\{2,\dots,p-1\}$.
\end{proposition}
\begin{proof}
The element $X_1(s)+\v w+\v h$ is equal to
$\frac {s-N_S} {a_1} e_1 +R(N_S)$. By Proposition \ref{pr:existencia}, the element $R(N_S)$ belongs to  $\Q^p_{\geq}$. Thus,  for every $s\geq N_S$ the element  $X_1(s)+\v w+\v h$ is also in  $\Q^p_{\geq}$. Similarly, it must hold  $X_p(s)+\v w'-\v h\in \Q^p_{\geq}$ for every $s\geq N_S$.
Thus,
$\overline{(X_1(s)+\v w+\v h) (X_p(s)+\v w'-\v h)}\subset \Q^p_{\geq}$.

For every $i\in\{2,\dots,p-1\}$,  the element
$X_1(s)+\v w+\v h+(p-2)\v {q_i}$ is equal to $\frac{s-N_S} {a_1}e_1 + R(N_S)+(p-2)\v {q_i}$.
By Proposition \ref{pr:existencia},  $R(N_S)+(p-2)\v {q_i}$ is in $\Q^p_{\geq}$. Thus for every $s\geq N_S$, we have $X_1(s)+\v w+\v h+(p-2)\v {q_i}\in \Q^p_{\geq}$.
Similarly, it can be obtained that $X_p(s)+\v w'-\v h+(p-2)\v {q_i}\in\Q^p_{\geq}$.
For every $X\in \overline{(X_1(s)+\v w+\v h)(X_p(s)+\v w'-\v h)}$, the element $X+(p-2)\v {q_i}$ belongs to the line segment with endpoints $(X_1(s)+\v w+\v h)+(p-2)\v {q_i}\in\Q^p_{\geq}$ and $(X_p(s)+\v w'-\v h)+(p-2)\v {q_i}\in\Q^p_{\geq}$. Hence, $X+(p-2)\v {q_i}\in\Q^p_{\geq}$.
\end{proof}

\begin{corollary}\label{cuadrados}
Let  $s\in\N$ be such that $s\geq N_S$. For every
$X\in \overline{(X_1(s)+\v w+\v h) (X_p(s)+\v w'-\v h)}$, the set
$\{X+\sum_{i=2}^{p-1}\lambda_i \v {q_i}\mid 0\leq\lambda_i\leq 1~
,\lambda_i\in\Q \}$ is contained in $ \Q^p_{\geq}$.
\end{corollary}
\begin{proof}
By Proposition \ref{pr:p}, the elements $X,X+(p-2)\v {q_2},\dots,X+(p-2)\v q_{p-1}$ are in $\Q^p_{\geq}. $
The smallest convex set  with respect the inclusion that contains the elements $X,X+(p-2)\v {q_2},\dots,X+(p-2)\v q_{p-1}$ is the set
$C=\{ X +\sum_{i=2}^{p-1} \mu_i (p-2)\v {q_i} \mid  0\leq \sum_{i=2}^{p-1}\mu_i \leq 1, \mu_i \in\Q_{\geq} \}$. Since $\Q^p_{\geq}$ is convex,  the set $C$ is a subset of  $\Q^p_{\geq}$. The set $C$ is equal to
$\{X+\sum_{i=2}^{p-1}\lambda_i\v {q_i}\mid 0\leq \sum_{i=2}^{p-1}\lambda_i\leq p-2,
\lambda_i\in\Q_{\geq}\}$ (just substitute $\mu_i(p-2)$ by $\lambda_i$), which clearly contains the set
$\{X+\sum_{i=2}^{p-1}\lambda_i\v {q_i}\mid 0\leq\lambda_i\leq 1,\lambda_i\in\Q\}$.
Thus, $\{X+\sum_{i=2}^{p-1}\lambda_i\v {q_i}\mid 0\leq\lambda_i\leq 1,\lambda_i\in\Q\}\subset \Q^p_{\geq}$.
\end{proof}

In order to complete our construction, we express $\E(s)$ as a union of three different sets.

\begin{definition}\label{d:zetas}
Let $S=\langle a_1,\dots, a_p\rangle$ be a numerical monoid.
For every $s\in \N$ such that $s\geq N_S$, define
\begin{itemize}
\item $\E_1(s)$  the set of elements
$x=(x_1,\dots,x_p)\in\E(s)$ verifying that $s/a_1+\l(\v w)< \l(x)\leq s/a_1$,
\item $\E_2(s)$  the set of elements
$x=(x_1,\dots,x_p)\in\E(s)$ verifying that $s/a_p+\l(\v w')-d\leq \l(x)\leq s/a_1+\l(\v w)+d$,
\item $\E_3(s)$  the set of elements
$x=(x_1,\dots,x_p)\in\E(s)$ verifying that $s/a_p\leq \l(x)< s/a_p+\l(\v w')$.
\end{itemize}

\end{definition}

For every $x=(x_1,\dots,x_p)\in\E(s)$ we have
$s=x_1a_1+\dots+x_pa_p$. This implies
$\frac s {a_1}=\frac 1 {a_1} (x_1a_1+\dots+x_pa_p)$, and using  $a_1<\dots<a_p$, we obtain
$\frac s {a_1}=x_1+x_2\frac {a_2}{a_1}+\dots+x_p\frac {a_p}{a_1}\geq x_1+\dots+x_p$.
Similarly, it can be proved that $\frac s {a_p}\leq x_1+\dots+x_p$.
Since $\l(x)=x_1+\dots+x_p$,
we obtain that $\frac s {a_1}=\l(X_1(s))\geq \l(x)\geq \l(X_p(s))=\frac s {a_p}$, and
hence $\E(s)=\E_1(s)\cup\E_2(s)\cup\E_3(s)$.

In the sequel, we use these sets to prove the periodicity of $\Delta(S)$, and to improve the algorithmic method that computes it.

Denote $\Delta(\E_i(s))=\{l_{j+1}-l_j\mid j=1,\dots,k-1\}$ with
$\{l_1<\dots<l_k\}$ the ordered set obtained from $\{\l(x)\mid x\in\E_i(s)\}
$ and  by $\lfloor x \rfloor$ the largest integer not greater than $x$.

\begin{theorem}\label{pr:d}
Let $s\in\N$ be such that $s\geq N_S$. Then
$\Delta(\E_2(s))=\{d\}$.
\end{theorem}
\begin{proof}
Take $s\in\N$ such that $s\geq N_S$.
Consider the set $K=\{X_p(s)+\v w'-\v h, X_p(s)+\v w', X_p(s)+\v w' +\v h, \dots,X_p(s)+\v w'+k\v h, X_1(s)+\v w+\v h\}$ with $k\in\N$ the maximum such that $\l(X_p(s)+\v w'+k\v h)\leq\l( X_1(s)+\v w+\v h)$. Denote by $\{l_0,l_1,\dots,l_{k+1},l_{k+2}\}$ the set $\l(K)$ where $l_i=\l(X_p(s)+\v w'+(i-1)\v h)$ with $i=0,\dots,k+1$ and $l_{k+2}=\l(X_1(s)+\v w+\v h)$.
 We have $l_{i}-l_{i-1}=d$ (therefore $l_i=l_0+id$)  for every $i=1,\dots,k+1$, and $l_{k+2}-l_{k+1}< d$. Since $\gcd(a_1,\dots,a_p)=1$, there exists $\gamma=(\gamma_1,\dots,\gamma_p)\in\Z^p$ such that $\sum_{i=1}^p\gamma_ia_i=s$, and thus $\gamma\in\H_s$.
 Let $\v v$ be a minimal length increase vector.
  We can find $\xi\in\Z$ such that $l_0\leq \l(\gamma+\xi\v v)=\l(\gamma)+\xi d<l_0+d=l_1$; since $\v v\in M$ the element $\gamma+\xi\v v$ is again in $\H_s$.
From the linearity of $\l$ and the fact that
$l_0=\l(X_p(s)+\v w'-\v h)\leq\l(\gamma+\xi\v v)<l_1=\l(X_p(s)+\v w')$, we can assert that there exists $\gamma^0\in \overline{(X_p(s)+\v w'-\v h)(X_p(s)+\v w')}\subset \H_s$ such that $\l(\gamma^0)=\l(\gamma+\xi\v v)$ (just solve the linear equation $\l(X_p(s)+\v w' - \lambda \v h)=\l(\gamma+\xi \v v)$ on $\lambda$ and take $\gamma^0=X_p(s)+\v w' -\lambda \v h$).
Since $\gamma+\xi\v v,\gamma^0\in\H_s$ and $\l(\gamma+\xi\v v)=\l(\gamma^0)$, the element $\gamma+\xi\v v-\gamma^0$ belongs to $\H_0\cap\V(E)$.
The set $\{\v {q_2},\dots,\v q_{p-1}\}$ is a basis of $\H_0\cap \V(E)$
 (see Remark \ref{r:base}), this implies that  there exist $\lambda_2,\dots,\lambda_{p-1}\in \Q$ such that $\gamma+\xi\v v-\gamma^0=\sum_{j=2}^{p-1}\lambda_j\v {q_j}$.
This leads to
$\gamma+\xi\v v-\sum_{j=2}^{p-1}\lfloor\lambda_j\rfloor\v {q_j}=
\gamma^0+\sum_{j=2}^{p-1}(\lambda_j-\lfloor\lambda_j\rfloor)\v {q_j}$.
The element $\gamma+\xi\v v-\sum_{j=2}^{p-1}\lfloor\lambda_j\rfloor\v {q_j}$ belongs to $\Z^p$, and
by Corollary \ref{cuadrados}  $\gamma^0+\sum_{j=2}^{p-1}(\lambda_j-\lfloor\lambda_j\rfloor)\v {q_j}\in\Q^p_{\geq}$. Thus,  $\gamma'^0=\gamma+\xi\v v-\sum_{j=2}^{p-1}\lfloor\lambda_j\rfloor\v {q_j}$ belongs to $\N^p$ and therefore $\gamma'^0\in \E(s)$.
Using that $\l(\v q_i)=0$, we obtain that  $\gamma'^0$ satisfies $l_0\leq \l(\gamma'^0)=\l(\gamma+\xi\v v-\sum_{j=2}^{p-1}\lfloor\lambda_j\rfloor\v {q_j})=\l(\gamma+\xi\v v)<l_0+d=l_1$. Now,
for every $i=1,\dots, k$, consider  the elements $\gamma+\xi \v v+i\v v$; they  verify that $l_i\leq \l(\gamma+\xi \v v+i\v v)=\l(\gamma+\xi \v v)+id<l_{i+1}$. If we proceed similarly, we obtain $\gamma'^i\in\E(s)$ fulfilling that $l_i\leq \l(\gamma'^i )=\l(\gamma+\xi\v v)+id<l_{i+1}$.
In this way we get a sequence of elements $\gamma'^0,\dots, \gamma'^k$ with lengths equal to $\l(\gamma'^0),\l(\gamma'^0)+d,\dots,\l(\gamma'^0)+kd$, respectively.
If $l_{k+1}=l_{k+2}$, then $X_p(s)+\v w'+k\v h= X_1(s)+\v w+\v h$ and we only have to check if $\l(\gamma'^0)+(k+1)d= l_{k+1}$. If so, then there exists $\gamma'^{k+1}$ such that $\l(\gamma'^{k+1})=\l(\gamma'^0)+(k+1)d$.
If
$l_{k+1}<l_{k+2}$, we check if $\l(\gamma'^0)+(k+1)d\leq l_{k+2}$. If so, then repeating the procedure with the element $\gamma+\xi \v v + (k+1)\v v$, we get an element $\gamma'^{k+1}$ with $\l(\gamma'^{k+1})=\l(\gamma'^0)+(k+1)d$.
The sequence obtained if formed by  $k$ or $k+1$ elements $\gamma'^0,\dots,\gamma'^{r}$ whose ordered set of lengths is  $\{\l(\gamma'^0),\l(\gamma'^0)+d,\dots,\l(\gamma'^0)+rd\}$, $\l(\gamma'^0)-l_0<d$ and $l_{k+2}-(\l(\gamma'^0)+rd)<d$.
Since $d=\min(\Delta(S))$, it is not possible to find more elements having different lengths, and therefore
$\Delta(\E_2(s))=\{d\}$.
\end{proof}

\begin{corollary}\label{cor:partes}
For all integer $s\geq N_S$, $\E_1(s)\cap \E_2(s)\neq \emptyset$ and $\E_2(s)\cap \E_3(s)\neq \emptyset$. Furthermore,  $\Delta(s)=\Delta(\E_1(s))\cup\{d\}\cup \Delta(\E_3(s))$.
\end{corollary}
\begin{proof}
From Definition \ref{d:zetas}, $\E_1(s)\cap \E_2(s)$ is formed by the elements $x\in\E(s)$ verifying that $s/a_1+\l(\v w)<\l(x)\leq s/a_1+\l(\v w)+d$.
Taking into account that $d=(s/a_1+\l(\v w)+d)-(s/a_1+\l(\v w))$, and using the proof of Theorem \ref{pr:d}, we can assert that there exists at least an element $\gamma'$ such that $s/a_1+\l(\v w)<\l(\gamma')\leq s/a_1+\l(\v w)+d$.
This implies that
$\E_1(s)\cap \E_2(s)\neq \emptyset$ and $\min\{\l(x)\mid x\in\E_1(s)\}$ is equal to $\max\{\l(x)\mid x\in\E_2(s)\}$.
Similarly, it can be proved that
 $\E_2(s)\cap \E_3(s)\neq \emptyset$ and
 $\max\{\l(x)\mid x\in\E_3(s)\}$ is equal to $\min\{\l(x)\mid x\in\E_2(s)\}$. In this way, we obtain that
$\Delta(s)=\Delta(\E_1(s))\cup\Delta(\E_2(s))\cup \Delta(\E_3(s))$.
Since $s\geq N_S$, by Theorem \ref{pr:d}, $\Delta(\E_2(s))=\{d\}$, and thus
$\Delta(s)=\Delta(\E_1(s))\cup\{d\}\cup \Delta(\E_3(s))$.
\end{proof}

The following result gives us the key to study the periodicity of $\Delta(S)$.

\begin{theorem}\label{pr:periodo}
Let $S=\langle a_1,\dots, a_p\rangle$ be a numerical monoid and
let $s\in\N$ be such that $s\geq N_S$.
Then $\E_1(s+a_1)=\{x+e_1\mid x\in \E_1(s)\}$ and
$\E_3(s+a_p)=\{x+e_p\mid x\in \E_3(s)\}$.
\end{theorem}
\begin{proof}
If $x\in \E_1(s)$, then $s/a_1+\l(\v w)< \l(x)\leq s/a_1$. Thus,
$(s+a_1)/a_1+\l(\v w)< \l(x)+1=\l(x+e_1)\leq (s+a_1)/a_1$, and therefore $x+e_1\in\E_1(s+a_1)$.

Let $y=(y_1,\dots,y_p)$ be an element of $\E_1(s+a_1)$.
Note that $(s+a_1)/a_1+\l(\v w)<\l(y)\leq (s+a_1)/a_1$, and thus
$s/a_1+\l(\v w)<\l(y-e_1)\leq s/a_1$.
If $y_1>0$,  then $y-e_1\in\E_1(s)$, and thus $y=(y-e_1)+e_1$ with $y-e_1\in\E_1(s)$.
Now assume  $y_1=0$.
The elements $y$ and $((s+a_1)/a_2)e_2$ are both in $\H_{s+a_1}$. Thus,
$y_2a_2+\dots+y_pa_p=((s+a_1)/a_2)a_2=s+a_1$.
Since $a_1<\dots<a_p$ and $y\in \Q^p_{\geq}$, we obtain that $\l(y)\leq (s+a_1)/a_2=\l(((s+a_1)/a_2)e_2)$. By Lemma \ref{ceros1}, $\l(X_2(s+a_1))=(s+a_1)/a_2\leq \l(X_1(s+a_1)+\v w)=(s+a_1)/a_1+\l(\v w)$. Thus $\l(y)\leq (s+a_1)/a_2\leq (s+a_1)/a_1+\l(\v w)$,  contradicting the fact that
$y\in \E_1(s+a_1)$.

Using
the same argument, we
also have  $\E_3(s+a_p)=\{x+e_p\mid x\in \E_3(s)\}$.
\end{proof}

\begin{corollary}
Let $S=\langle a_1,\dots,a_p\rangle$ be a numerical monoid. Then
$\Delta(S)=\bigcup_{
s\in S, s\leq N_S+a_p-1
}\Delta(s).$
Furthermore, $\Delta(S)$ is periodic from $N_S$ with period $\lcm(a_1,a_p)$.
\end{corollary}
\begin{proof}
From Theorem \ref{pr:periodo}, it follows that $\Delta(\E_1(s+a_1))=\Delta(\E_1(s))$ and $\Delta(\E_3(s+a_p))=\Delta(\E_3(s))$ for all $s\geq N_S$. By Corollary \ref{cor:partes}, we obtain  $\Delta(S)=\cup_{s=0}^{N_S+a_p-1}\Delta(s)$.

Theorem \ref{pr:periodo} gives us that in the set $\{n\in\N\mid n\geq N_S\}$ the function $\Delta(\E_1(s))$ is periodic with period $a_1$ and  $\Delta(\E_3(s))$ is periodic with period to $a_p$. Hence, $\Delta(s)$ is periodic with period $\lcm(a_1,a_p)$.
\end{proof}

\begin{remark}
If $s\geq N_S$, then $d\in\Delta(s)$. Thus,  $\E(s)\neq\emptyset$, and so
the number $N_S-1$ is greater than the Frobenius number of $S$.
\end{remark}

\renewcommand{\gcd}{\textrm{gcd}}

\section{Formulation of $N_S$}\label{sec:formulas}
From Proposition \ref{pr:existencia}, there exists  $N_S\in\N$ fulfilling that
$R(N_S)$ , $R'(N_S)$,
$P_2(N_S)+\v h+(p-2)\v {q_i}$
and
$P_{p-1}(N_S)-\v h+(p-2)\v {q_i}$ belong to $\Q^p_{\geq}$.
To compute $N_S$,  we proceed as follows for every $i\in\{2,\dots,p-1\}$:
\begin{enumerate}
\item Consider the element $P_2(s)+\v h+(p-2)\v {q_i}$. This element is equal to
\[\begin{multlined}
\left(
\frac{(p-2) \left(a_i-a_p\right)}{\gcd \left(a_i-a_1,a_1-a_p,a_p-a_i\right)}+\frac{d a_p}{a_p-a_1}+\frac{s \left(a_2-a_p\right)}{a_2 \left(a_1-a_p\right)}
\right)
e_1\\
+\frac{(p-2) \left(a_p-a_1\right)}{\gcd \left(a_i-a_1,a_1-a_p,a_p-a_i\right)}
e_i\\
+
\left(
\frac{(p-2) \left(a_1-a_i\right)}{\gcd \left(a_i-a_1,a_1-a_p,a_p-a_i\right)}+\frac{a_1 d}{a_1-a_p}+\frac{\left(a_1-a_2\right) s}{a_2 \left(a_1-a_p\right)}
\right)
e_p.
\end{multlined}\]
Its $i$th coordinate is always positive and  increasing the value of $s$ we may obtain an element of $\Q^p_{\geq}$.
Besides, if $P_2(s)+\v h+(p-2)\v {q_i}\in\Q^p_{\geq}$, then for every $s'\geq s$ the element
$P_2(s')+\v h+(p-2)\v {q_i}$ is also in $\Q^p_{\geq}$.
Note also that, since
$((p-2)\v {q_i})_1=\frac{(p-2) \left(a_i-a_p\right)}{\gcd \left(a_i-a_1,a_1-a_p,a_p-a_i\right)}\leq 0$ and
$((p-2)\v {q_i})_p=\frac{(p-2) \left(a_1-a_i\right)}{\gcd \left(a_i-a_1,a_1-a_p,a_p-a_i\right)}\leq 0$,
if $P_2(s)+\v h+(p-2)\v {q_i}\in\Q^p_{\geq}$, then  $R(s)=P_2(s)+\v h\in\Q^p_{\geq}$.
\item  Solve the linear equation $(P_2(s)+\v h+(p-2)\v {q_i})_{p}=0$ on $s$. Its solution is
\begin{equation}\label{s1}
S_i=
-\frac{a_2 \left(a_1 d \gcd \left(a_i-a_1,a_1-a_p,a_p-a_i\right)+(p-2) \left(a_1-a_i\right) \left(a_1-a_p\right)\right)}{\left(a_1-a_2\right) \gcd \left(a_i-a_1,a_1-a_p,a_p-a_i\right)}
.
\end{equation}
    The element $P_2(S_i)+\v h+(p-2)\v {q_i}$ is equal to
    \[\begin{multlined}
-\frac{a_2 d \gcd \left(a_i-a_1,a_1-a_p,a_p-a_i\right)+(p-2) \left(a_2-a_i\right) \left(a_1-a_p\right)}{\left(a_1-a_2\right) \gcd \left(a_i-a_1,a_1-a_p,a_p-a_i\right)}
e_1\\
+
\frac{(p-2) \left(a_p-a_1\right)}{\gcd \left(a_i-a_1,a_1-a_p,a_p-a_i\right)}
e_i,
\end{multlined}\]
it is an element of $\Q_{\geq}^p$, and therefore $R(s)\in\Q^p_{\geq}$ and $P_2(s)+\v h+(p-2)\v {q_i}\in\Q^p_{\geq}$ for all $s\geq S_i$, $s\in\N$.
\item
We proceed similarly with the elements $P_{p-1}(s)-\v h+(p-2)\v {q_i}$ for every  $i=2,\dots,p-1$.
In this case the solution of $(P_{p-1}(s)-\v h+(p-2)\v {q_i})_1=0$ is
\begin{equation}\label{s2}
S'_i=
\frac{a_{p-1} \left((p-2) \left(a_1-a_p\right) \left(a_p-a_i\right)-d a_p \gcd \left(a_i-a_1,a_1-a_p,a_p-a_i\right)\right)}{\left(a_{p-1}-a_p\right) \gcd \left(a_i-a_1,a_1-a_p,a_p-a_i\right)}.
\end{equation}
It is easy to check that $P_{p-1}(S'_i)-\v h+(p-2)\v {q_i}$ is again an element of $\Q_{\geq}^p$ and $R'(s),P_{p-1}(s)-\v h+(p-2)\v {q_i}\in\Q^p_{\geq}$ for all  $s\geq S'_i$, $s\in \N$.
\item
We set $N_S$ as the least integer greater than or equal to all the $2(p-2)$ solutions obtained:
\begin{equation}\label{forumula:NS}
N_S=\lceil \max(\{S_i\mid i=2,\dots,p-1\}\cup\{S'_i\mid i=2,\dots,p-1\})\rceil,
\end{equation}
where $\lceil x \rceil$ represents  the smallest integer not less than $x$. This value  fulfills  the same properties than $N_S$ in Proposition \ref{pr:existencia}.
\end{enumerate}

If we see our bound as a rational function on the variable $a_p$,  the numerator has degree $2$ and the denominator has degree $1$. So, our bound increase linearly on $a_p$.
On the contrary, the bound appearing in \cite{Chapman-Hoyer-Kaplan} is a polynomial of degree $2$ on $a_p$.
The results of Table \ref{t1} confirm the expected behaviors for both bounds.

\begin{example}
Let
$S=\langle 15,17,27,35\rangle $. In this case
the group $M$ is generated by
$\{(-13,1,4,2),(-9,0,5,0),(-7,0,0,3)\}$. The lengths of these vectors are $-6$, $-4$ and $-4$. Since $\gcd(6,4,4)=2$, the value of $d$ is equal to $2$. The solutions are the following
\[\begin{array}{l}
S_2=-\frac{a_2 \left(a_1 d \gcd \left(a_2-a_1,a_1-a_4,a_4-a_2\right)+\left(a_1-a_2\right) \left(a_1-a_4\right) (p-2)\right)}{\left(a_1-a_2\right) \gcd \left(a_2-a_1,a_1-a_4,a_4-a_2\right)}=595,\\
S_3=-\frac{a_2 \left(a_1 d \gcd \left(a_3-a_1,a_1-a_4,a_4-a_3\right)+\left(a_1-a_3\right) \left(a_1-a_4\right) (p-2)\right)}{\left(a_1-a_2\right) \gcd \left(a_3-a_1,a_1-a_4,a_4-a_3\right)}=1275,\\
S'_2=\frac{a_3 \left(\left(a_1-a_4\right) \left(a_4-a_2\right) (p-2)-a_4 d \gcd \left(a_2-a_1,a_1-a_4,a_4-a_2\right)\right)}{\left(a_3-a_4\right) \gcd \left(a_2-a_1,a_1-a_4,a_4-a_2\right)}=\frac{5805}{4}=1451.25,\\
S'_3=\frac{a_3 \left(\left(a_1-a_4\right) \left(a_4-a_3\right) (p-2)-a_4 d \gcd \left(a_3-a_1,a_1-a_4,a_4-a_3\right)\right)}{\left(a_3-a_4\right) \gcd \left(a_3-a_1,a_1-a_4,a_4-a_3\right)}=\frac{2025}{4}=506.25,
\end{array}\]
and thus
$N_S=\lceil\max(595,1275,1451.25,506.25)\rceil=1452$.
\end{example}

\section{Computation of $\Delta(S)$}\label{sec:compativa}

Lemma \ref{lema:d} and
Formula \ref{forumula:NS} give us $d$ and $N_S$, respectively.
Both values are needed to compute $\Delta(S)$.
The next step in our algorithm is the computation of
$\E(N_S+a_p-1),\dots,\E(N_S+a_p-1-a_1-1)$; each of these sets are the
nonnegative integer solutions of a Diophantine equation.
From these sets we define the sets
\[\Omega(s)=\{(x_1,\l(x))\mid x\in\E(s),~x_1=\max\{y_1\mid y\in\E(s),~\l(y)=\l(x)\}\}\]
for all $s=N_S+a_p-1,\dots, N_S+a_p-1-a_1-1$.
The second coordinate $\l(x)$ is used in the computation of $\Delta(s)$; just project the second coordinate of the elements of $\Omega(s)$, order this set obtaining $\{l_1<\dots<l_{t_s}\}$ and compute $\{l_{i+1}-l_i\mid i=2,\dots,t_s\}$.
It is straightforward to prove that
for every $s\in\N$ the set $\E(s-a_1)$ is equal to $\{(x_1,\dots,x_p)-e_1\mid (x_1,\dots,x_p)\in\E(s),~x_1\geq 1\}$, and therefore
$\Omega(s-a_1)=\{(x_1-1,\l(x)-1)\mid (x_1,\l(x))\in\Omega(s),~x_1> 0\}$.
This allows us to obtain all the sets $\Omega(s)$ with $s\in S$ and $s< N_S+a_p-1-a_1-1$
from the sets $\Omega(N_S+a_p-1),\dots,\Omega(N_S+a_p-1-a_1-1)$.

Algorithm \ref{algorithm} collects the different improvements made in this work and computes the Delta set of a given numerical monoid. In addition, note that the computation of the sets $\E(*)$ and $\Omega(*)$ in steps  2 to 5 could be done in a parallel way.
Our implementation of this algorithm in \cite{programita}  takes into account this fact.

\begin{algorithm}\label{algorithm}
The input is $S=\langle a_1,\dots,a_p\rangle$ a numerical monoid.
The output is the set $\Delta(S)$.
\begin{enumerate}
\item Using Lemma \ref{lema:d}, compute $d$.
\item Using Equations (\ref{s1}), (\ref{s2}) and (\ref{forumula:NS}), compute $N_S$.
\item Compute the sets $\E(N_S+a_p-1),\dots,\E(N_S+a_p-1-a_1-1)$.
\item Compute the sets $\Omega(N_S+a_p-1),\dots,\Omega(N_S+a_p-1-a_1-1)$.
\item Compute $\{\Omega(s)\mid s\in S,~s\leq N_S+a_p-1\}$, using the sets of the preceding step.
\item Compute $\Upsilon=\{\Delta(s)\mid s\in S,~s\leq N_S+a_p-1\}$,
using the sets of the preceding step.
\item Return $\Delta(S)=\cup_{ \Psi \in \Upsilon}\Psi$.
\end{enumerate}
\end{algorithm}

Solve a problem by using an exhaustive search usually have two computational disadvantages: the common high size of the bounds and the complexity to check the corresponding properties; but these computational troubles have to be assumed if there exists no other way to solve the problem.
 The algorithm method introduced in this work and the method that appears in \cite{Chapman-Hoyer-Kaplan} are based on computing the set of factorizations of some bounded elements belonging to a numerical monoid. The size of our bound $N_S+a_p-1$ (equation (\ref{forumula:NS})) seems to be better than the bound of \cite[Theorem 1]{Chapman-Hoyer-Kaplan}.

Another remarkable fact is that for each element
in the set $\{N_S+a_p-1,\dots,N_S+a_p-1-a_1-1\}$, it is necessary  to compute its set of factorizations by solving its defining equation. This is a high computational complexity problem ($\mathrm{NP}$-complete problem) although there exist some specific algorithms for solving a unique Diophantine equation (see \cite{Fortenbacher}). Anyway, since our  bound reduces the size of the checking region, Algorithm \ref{algorithm} needs to do less high complexity computations than the algorithm obtained from \cite[Corollary 3]{Chapman-Hoyer-Kaplan}. So the usual computational disadvantages are partly avoided.

The above computational considerations are reflected in the following examples.

\begin{example}
We compare our bound with the bound of \cite{Chapman-Hoyer-Kaplan}. Using the functions
{\tt ChapmanBound} and {\tt Bound} of \cite{programita} we obtain Table \ref{t1}. For the semigroup $S=\langle 15,16,27\rangle$ these functions can be used as follows:
\begin{verbatim}
In[1]:= Bound[{15, 16, 27}]
Out[1]= 446
...
In[2]:= ChapmanBound[{15, 16, 27}]
Out[2]= 70224
\end{verbatim}

\begin{table}[ht]
\caption{Comparison with bound of \cite{Chapman-Hoyer-Kaplan}}
\centering
\begin{tabular}{c| c| c}
\hline
Semigroup & Bound of \cite{Chapman-Hoyer-Kaplan} & $N_S+a_p-1$\\
\hline $\langle 15,16,27\rangle$ & $70224$ & $446$\\
$\langle 37, 59, 101\rangle $ & $3613337$ & $2208$\\
$\langle 201, 451, 577\rangle$ & $900996525$ & $89119$\\
$\langle 15, 17, 27, 35\rangle $ & $166855$ & $1466$\\
$\langle 100, 121, 142, 163, 284\rangle$ & $97605860$ & $201052$\\
$\langle 1001, 1211, 1421, 1631, 2841\rangle$ & $97744425121$ & $2064141$\\
\end{tabular}
\label{t1}
\end{table}

\noindent
Clearly, our bound is lower that  the bound of Theorem 1 of \cite{Chapman-Hoyer-Kaplan}.

\end{example}

\begin{example}

To compute the Delta set of a numerical monoid can be used the function {\tt DeltaSNParallel} (see \cite{programita}) as follows:
\begin{verbatim}
In[3]:= AbsoluteTiming [DeltaSNParallel[{4, 6, 15}]]
Out[3]={0.018001, {1, 2, 3}}
\end{verbatim}
The returned value is $\{0.018001,\{1,2,3\}\}$ where $0.018001$ is the time required for its computation, and $\{1,2,3\}$ is the Delta set of the numerical monoid $S=\langle 4,6,15\rangle$.
Table \ref{t2} contains the computation of $\Delta(S)$ for some numerical monoids. All these examples have been done in an Intel Core i7 with 16 GB of main memory using the parallel version of the program.

\begin{table}[ht]
\caption{Computation of $\Delta(S)$}
\centering
\begin{tabular}{c|c|c|c|c}\hline
Semigroup & Bound of \cite{Chapman-Hoyer-Kaplan} & $N_S+a_p-1$ &  $\Delta(S)$  & Time (seconds)\\
\hline $\langle 4,6,15\rangle$ & $8124$ & $81$ & $\{1, 2, 3\}$ & $0.018$\\
$\langle 4, 6, 199\rangle $ & $1425660$ &$1185$ & $\{1,2,\dots,65\}$ & $0.152$\\
$\langle 11,37,52,93\rangle$ & $2560511$ & $5952$ & $\{1,\dots,21\}$ & $5.205$\\
$\langle 51, 53, 55, 117\rangle$ & $5806839$ &$9749$ & $\{2,4,6\}$ & $6.962$\\
$\langle 11,53,73,87\rangle$ & $3209839$ & $14391$ & $\{2,4,6,8,10,22\}$ & $33.098$\\
$\langle 11, 53, 73, 81\rangle $ & $2782447$ & $6599$ & $\{2,4,6,8,10,12\}$ & $3.713$\\
$\langle 7,15,17,18,20\rangle $ & $60105$ & $1941$ & $\{1,2,3\}$ & $151.668$\\
$\langle 10,17,19,25,31 \rangle $ & $163540$ & $1189$ & $\{1,2,3\}$ & $8.629$\\
$\langle 10,17,19,21,25\rangle $ & $106420$ & $2031$ & $\{1,2\}$ & $102.656$ \\
$\langle 7,19,20,25,29\rangle$ & $159923$ & $3900$ & $\{1,2,3,5\}$ & $878.702$\\
$\langle 31,73,77,87,91 \rangle$ & $6047393$ &$31394$ & \{2, 4, 6\} & $24012.245$
\end{tabular}
\label{t2}
\end{table}
\end{example}

\end{document}